\documentclass[preprint,11pt]{amsart}
 \usepackage{amssymb}
\usepackage{amsmath,amsfonts,amscd}
\usepackage{mathrsfs,hyperref, enumerate}

\usepackage[a4paper,bindingoffset=0.2in,left=1in,right=1in,top=1in,bottom=1in,footskip=.25in]{geometry}

\usepackage{amsmath,amsfonts,amsthm,amssymb,amscd}
\usepackage{bbm}
\usepackage{mathrsfs}
\usepackage{color}
\usepackage{hyperref}
\usepackage{tikz-cd}
\usepackage{braket}

\usetikzlibrary{matrix}
\usepackage{graphicx}
\usepackage{enumerate}
\usepackage{changepage}

\def\A{\mathbb{A}}

\def\F{\mathbb{F}}
\def\P{\mathbb{P}}

\newcommand{\CalC}{\mathcal{C}}
\newcommand{\CalF}{\mathcal{F}}

\theoremstyle{definition}
\newtheorem{thm}{Theorem}
\numberwithin{thm}{section}
\newtheorem{cor}[thm]{Corollary}
\newtheorem{prop}[thm]{Proposition}
\newtheorem{lem}[thm]{Lemma}

\newcommand{\gam}{\gamma}

\newcommand{\Hom}{\operatorname{Hom}}
\newcommand{\End}{\operatorname{End}}
\newcommand{\Ker}{\operatorname{Ker}}

\newcommand{\eps}{\epsilon}
\newcommand{\Tr}{\operatorname{Tr}}
\newcommand{\abs}[1]{\left| #1 \right|}
\def\sl{\mathfrak{sl}}

\newcommand{\ip}[1]{\langle #1 \rangle}
\newcommand{\tit}[1]{\textit{#1}}

\newcommand{\trm}[1]{\textrm{#1}}
\newcommand{\al}{\alpha}

\def\F{\mathbb{F}}

\newcommand{\lam}{\lambda}

\newcommand{\tr}{\textrm{Tr}}

\newcommand{\aff}{\textrm{aff}}
\newcommand{\adj}{\textrm{adj}}
\newcommand{\GL}{\textrm{GL}}
\newcommand{\sub}{\subset}
\newcommand{\id}{\mathrm{I}}
\newcommand{\ol}[1]{\overline{#1}}
\begin{document}

\title{Additive Decompositions by Conjugacy Classes in $M_n(\mathbb{F}_q)$}

\begin{abstract}
Let $n \geq 2$ be a positive integer, and $q$ be a prime power.  We study the \tit{number} of additive decompositions of nonscalar matrices in the matrix ring $M_n(\F_q)$ as sums of elements from two prescribed conjugacy classes. Let $z \in M_n(\F_q)$ be nonscalar. We show that, except for the case $(n,q,\tr(z)) = (2,2,1)$, there exist conjugacy classes $X, Y  \sub M_n(\F_q)$ such that the characteristic polynomial of $X$ is irreducible of degree $n$ and the characteristic polynomial of $Y$ is of the form $ (T- \lambda) h(T)$, where $\lambda \in \F_q$, $h$ is irreducible of degree $n-1$ and $h(\lambda) \neq 0$. These classes can be chosen so that $\tr(z) = \tr(X) + \tr(Y)$. For such $X$ and $Y$, let 
$$
N_{X,Y}(z) = \# \Set{ (x,y) \in X \times Y : x + y = z}.
$$
We prove the following estimate:
$$
\left| N_{X,Y}(z) - q^{(n-1)^2}  \right| \leq 42 q^{(n-1)^2 -1}.
$$
Thus, for nonscalar matrices with matching trace, the number of such additive decompositions is \tit{approximately} the same, namely $q^{(n-1)^2}$,  with an absolute error constant independent of $n$ and $q$.
\end{abstract}

\author{Krishna Kishore and Sunil Kumar Mallick}

\email{kishorekrishna@iittp.ac.in}
\address{Krishna Kishore,
Department of Mathematics and Statistics,
Indian Institute of Technology Tirupati,
Tirupati, Andhra Pradesh
India, 517619.}

\email{mallicksunil200114@gmail.com}
\address{Sunil Kumar Mallick,
Department of Mathematics and Statistics,
Indian Institute of Technology Tirupati,
Tirupati, Andhra Pradesh
India, 517619.}



\maketitle

\section{Introduction}\label{intro}

Larsen, Shalev, and Tiep \cite{LST}   proved the following assertion on their way to proving Waring's problem for finite simple groups; also see \cite{Sh}: 
Let $n \geq 2$ be a positive integer, $q$ a prime power, and $F := \F_q$ be the finite field with $q$ elements. Let $X$ be a conjugacy class of an element of $\GL_n(F)$ whose characteristic polynomial is irreducible of degree $n$ and $Y$ be a conjugacy class of an element whose characteristic polynomial is of the form $(T- \lambda) h(T)$ where $\lambda \in F$,  $h(T)$ is irreducible of degree $n-1$ and $h(\lambda) \neq 0$. Then every non-identity element whose determinant is $\det(X) \cdot \det(Y)$ can be written in the same number of ways as a product $x \cdot y$ with $x \in X$ and $y \in Y$, if $q$ is sufficiently large. In this article, we ask whether  an analogous statement holds in $M_n(F)$, and prove that if $X$ and $Y$ are fixed conjugacy classes in $M_n(F)$ of the prescribed type as above, then every \tit{nonscalar} element of $M_n(F)$ with  trace equal to $\tr(X) + \tr(Y)$ can be expressed in \textit{approximately} the same number of ways as a sum $x + y$ where $x \in X$ and $y \in Y$; see Theorem \ref{mainthm1} below for a precise formulation.

The present result was motivated in part by the Matrix Waring Problem. Our earlier work \cite{Ki, KS, KVZ} established \tit{existence} of Waring decompositions for sufficiently large finite fields. The purpose of the present paper is different: rather than merely establishing existence, we obtain a uniform asymptotic formula for the \textit{number} of additive decompositions with the summands constrained to lie in the prescribed conjugacy classes.

We now state the results.  Let $X$ be a conjugacy class in $M_n(F)$ whose characteristic polynomial $\chi_X(T)$ is irreducible of degree $n$, and $Y$ be a conjugacy class in $M_n(F)$ whose characteristic polynomial $\chi_Y(T) := (T- \lambda) h(T)$, where $\lambda \in F$, $h(T)$ is irreducible of degree $n-1$ and $h(\lambda) \neq 0$. For $z \in M_n(F)$, let
$$
N_{X, Y} (z) := \# \Set{(x,y) \in X \times Y : x + y = z},
$$
namely the number of representations of $z$ as a sum $x + y$ such that $x \in X$ and $y \in Y$. Let $\tr(z)$ and $\tr(X)$ denote the trace of $z$ and the trace of $X$ respectively. If $z = x + y$ with $x \in X$ and $y \in Y$, then necessarily $\tr(z)  = \tr(X) + \tr(Y)$. Thus $N_{X, Y}(z) = 0$ if $\tr(z) \neq \tr(X) + \tr(Y)$, and so we are interested only in the case where $\tr(z) = \tr(X) + \tr(Y)$, i.e. when $z$ has \textit{matching trace}. However, this is not a serious restriction at all, because we eventually prove that except when $(n,q,\tr(z)) = (2,2,1)$,  given any nonscalar $z \in M_n(F)$ there exist conjugacy classes $X$ and $Y$ of the above type such that $\tr(z) = \tr(X) + \tr(Y)$; see Corollary \ref{cor2} below.

First, we can immediately resolve the case where $z$ is a scalar matrix. Indeed, suppose that $z \in M_n(F)$ is a \textit{scalar} matrix with trace $\tr(X) + \tr(Y)$. Clearly, $N_{X, Y}(z)$ is equal to  $\# \Set{ X \cap (aI-Y) }$. We claim that the set $X \cap (aI - Y)$ is empty. Suppose, on the contrary that there exist elements $x \in X, y \in Y$ such that  $x = aI -y$. Since $\lam \in F$ is an eigenvalue of $y$, $a- \lam \in F$ is an eigenvalue of $aI - y$, hence that of $x$, contrary to the hypothesis $\chi_X(T)$ is irreducible over $F$.  Thus $N_{X, Y}(z) = 0$ for scalar $z$ \tit{even} with matching trace,  and so we are interested in the nontrivial part of the following results:

\begin{thm}\label{main_theorem_n=2}
Let $X$ be a conjugacy class in $M_2(F)$ whose characteristic polynomial is irreducible of degree $2$, and $Y$ be a conjugacy class in $M_2(F)$ whose characteristic polynomial is of the form $(T - \lambda)(T- \mu)$, where $\lambda, \mu \in F$ and $\lambda \neq \mu$. Let $z \in M_2(F)$ with trace $\textrm{Tr}(X) + \textrm{Tr}(Y)$. Then $N_{X, Y}(z) = 0$ if $z$ is a scalar matrix, otherwise $N_{X, Y}(z) \in \Set{q-1,q,q+1}$. 
\end{thm}

\begin{thm}\label{mainthm1}
Let  $n \geq 2$. Let $X$ be a conjugacy class in $M_n(F)$ whose characteristic polynomial is irreducible of degree $n$, and $Y$ be a conjugacy class in $M_n(F)$ whose characteristic polynomial is of the form $(T - \lambda)h(T)$, where $\lambda \in F$ and $h(T)$ is an irreducible polynomial over $F$ of degree $n-1$, and $h(\lambda) \neq 0$ (so that if $n =2$ the linear factors are distinct). Let $z \in M_n(F)$ with trace $\textrm{Tr}(X) + \textrm{Tr}(Y)$. Then $N_{X, Y}(z) =0$ if $z$ is a scalar matrix, otherwise
$$
\abs{N_{X, Y} (z) - q^{(n-1)^2} } \leq 42 q^{(n-1)^2-1}.
$$
In other words, $N_{X, Y}(z) = q^{(n-1)^2} + O\left(q^{(n-1)^2 -1} \right)$, with an absolute implied constant. 
\end{thm}

\noindent Combining Theorem \ref{mainthm1} with a trace-realization argument, Lemma \ref{tracerealization}, yields the following corollary:

\begin{cor}\label{cor2}
Let $n \geq 2$, and let $z \in M_n(F)$ be nonscalar. Put $t = \tr(z)$, and suppose $(n,q, t) \neq (2,2,1)$.  Then there exist conjugacy classes $X, Y \sub M_n(F)$ such that the characteristic polynomial of $X$ is irreducible of degree $n$, and the characteristic polynomial of $Y$ is of the form $(T - \lambda) h(T)$, where $h$ is an irreducible polynomial over $F$ of degree $n-1$ and $h(\lambda) \neq 0$ (so that if $n = 2$, we require the two linear factors be distinct.) Moreover, $\tr(z)= \tr(X) + \tr(Y)$ and the number $N_{X,Y}(z)$ of such decompositions satisfies the following estimate:
$$
\left| N_{X,Y}(z) - q^{(n-1)^2} \right| \leq 42 q^{(n-1)^2 -1}.
$$
\end{cor}
\noindent In other words, the quantities $N_{X,Y}(z)$ as $z$ varies over all nonscalar matrices are \tit{approximately} equal, if $q$ is sufficiently large. Moreover, note that in the case where $n = 2$, Theorem \ref{main_theorem_n=2} is stronger than Theorem \ref{mainthm1}. Corollary \ref{cor2} also shows that the only obstruction to choosing conjugacy classes of the prescribed types with matching trace is $(n,q, \tr(z)) = (2, 2, 1)$. 

One might wonder whether it is possible for $N_{X,Y}(z)$ can be equal for all nonscalar $z$. That is not the case. Indeed, consider the conjugacy action of $\GL_n(F)$ on $M_n(F)$. For a fixed $x_{0}\in X$, the stabilizer of $x_{0}$ is the centralizer of $x_{0}$ in $\GL_n(F)$ and is equal to the multiplicative group of the finite field with $q^n$ elements. Therefore, 
\begin{align*}
\#{X} &= \frac{\#{\GL_n(F)}}{q^{n}-1} = q^{\frac{n(n-1)}{2}} \prod_{k=1}^{n-1}(q^{k}-1).
\end{align*}
Similarly, the centralizer of $y_{0}\in Y$ in $\GL_n(F)$ is the product of the multiplicative group of finite field with $q$ elements and the multiplicative group of finite field with $q^{n-1}$ elements, so that its order is $(q-1)(q^{n-1}-1)$. Hence 
\begin{align*}
\#Y = \frac{\#{\GL_n(F)}}{(q-1)(q^{n-1}-1)} = q^{\frac{n(n-1)}{2}}\prod_{k=2, k\ne n-1}^{n}(q^{k}-1).
\end{align*}
It follows that, $$\#{X\times Y}=q^{n(n-1)}(q-1)(q^{n-1}-1)(q^{n}-1)\prod_{k=2}^{n-2}(q^{k}-1)^{2}.$$

\noindent Now, observe  that the set $Z := \Set{ M \in M_n(F) : \tr(M)=\tr(X)+ \tr(Y) }$ is an affine subspace of $\A^{n^2}_{F}$ of dimension $n^{2}-1$, so that its size is $q^{n^{2}-1}$. Consider the addition map $\phi:X\times Y \longrightarrow Z$ given by $\phi(x,y)=x+y$. Suppose, without loss of generality, that the map is surjective; otherwise, the conclusion is immediate. If each fiber of $\phi$ has the same cardinality then the cardinality of $Z$, which is $q^{n^{2}-1}$, must divide the $\# X\times Y$. But $q^{n^{2}-1}$ does not divide $q^{n(n-1)}$, since $n^{2}-1-n(n-1)=n-1>0$ for all $n \geq 2$.   Thus, there exist elements $z,z' \in M_n(F)$ such that $\tr(z)=\tr(z')=\tr(X)+\tr(Y)$ but $\# \Set{(x,y)\in X\times Y\mid z=x+y } \neq \# \Set{(x,y)\in X\times Y\mid z'=x+y}$.

We now give a brief description of the strategy of proofs of the above results. First consider Theorem \ref{main_theorem_n=2}. 
When $q$ is odd, we first observe that every matrix in $M_2(F)$ can be expressed uniquely as a sum of a scalar matrix and a traceless matrix and then proceed to replace the set $\Set{ (x,y) \in X \times Y : x + y = z}$ with the product $X' \times Y'$ of conjugacy classes $X', Y'$ in $\sl_2(q)$, $2 \times 2$ traceless matrices equipped with a quadratic form $Q: \sl_2(q) \to F$ and the associated bilinear form $\ip{A,B} := Q(A+B) - Q(A) - Q(B)$. So the triple $(x,y,z)$ is replaced by $(X_0, Y_0, Z_0)$ with $Z_0 = X_0 + Y_0$. After writing $X_0 = Z_0 - Y_0$, the conditions $Q(X_0) = r$ and $Q(Y_0) = s$ reduce to a quadratic equation $Q(Y_0) = s$ together with the linear equation $\ip{Z_0, Y_0} = Q(Z_0) + s - r$. Thus the required count $N_{X,Y}(z)$ becomes the number of points in the intersection of the quadratic surface and the plane; we count them in terms of the Dirichlet quadratic character $\chi:F \to \Set{-1,0,1}$  with the aid of the well-known orthogonality relation $\sum_{a \in F} \chi(a) = 0$.

The case when $q$ is even is slightly more delicate, since the above decomposition of $M_2(F)$ is not available in this case. Instead, we work directly with a  characterization of the set $X \times Y$ in terms of conditions imposed on the trace and the determinant of the characteristic polynomials of $X$ and $Y$. The trace condition defines an affine space $\A_{F}^3$, and one of the determinant conditions defines an affine surface  and the other defines a plane in \tit{this} affine space. We prove that their intersection in the associated projective space  is a nondegenerate conic, so that the intersection is isomorphic to the projective line $\P^1_{\overline{F}}$. Finally, we consider its $F$-rational points, which are at most $q+1$ in number.

Now consider the main Theorem \ref{mainthm1}. Let $V := F^n$ be the $n$-dimensional vector space over $F$, and $\P(V)$ be the associated projective space. We use the counting-in-two-ways strategy. Given a nonscalar $z$ in $M_n(F)$ and a choice of $X$ and $Y$ with prescribed characteristic polynomials, consider the map 
\begin{align*} \pi_z : \Set{ (x,y) \in X \times Y : x + y  =z} &\to \P(V) \\
(x,y) &\mapsto L_y
\end{align*}
where $L_y$ is the unique $\lambda$-eigenline of $y$, and count $N_{X, Y}$ as a sum of the cardinalities of the fibers $\pi_z^{-1}(L)$, as $L$ varies over $\P(V)$. We then show that precisely those $L_y \in \P(V)$ contribute to the sum which are \tit{not} $z-$stable. Let $z = \begin{pmatrix} a & b \\ c & A \end{pmatrix}$ and $y = \begin{pmatrix} \lambda & r \\ 0 & B \end{pmatrix} \in Y$ where $\lambda \in F$, $V = L_y \oplus W$, i.e. $W$ is a fixed complement of $L_y$ in the vector space $V$, $r \in W^*, c \in W$, and $A, B \in \End_{F}(W)$. Thus the characteristic polynomial of $y$ is $\chi_y(T) =(T- \lam) \chi_B(T)$, and it follows from the equivalence in Lemma \ref{conjugacy_characteristic_equivalence} that $y \in Y$ if and only if $\chi_B(T) = h(T)$; see \S 2 for notation. Let 
$$
\CalC_h = \Set{B \in \End_{F}(W) : \chi_B(T) = h(T)}.
$$
Then we observe that among the lines which are not $z$-stable, the cardinality of the fiber is given by
$$
\#\pi_z^{-1}(L) = \# \Set{B \in \CalC_h : c \trm{ is cyclic for } A - B}.
$$
Instead of counting the matrices $B \in \CalC_h$ for which $c$ is a cyclic vector for $A- B$, we count those elements $B \in \CalC_h$ for which $c$ is \tit{not} a cyclic vector in $A- B$, and subtract it from the cardinality of $\CalC_h$. Based on a result of Ram \cite{Ra} we  obtain the following crucial estimate
: for every non-$z$-stable line $L$, we have
$$
\abs{\# \pi_z^{-1}(L)  - q^{(n-1)(n-2)}} \leq 18q^{(n-1)(n-2) -1}.
$$
This leads to the desired estimate in the theorem.

In passing, it is natural to ask whether one could use the fundamental result of Lang-Weil \cite{LW}, by interpreting our problem as counting points on a geometrically irreducible affine variety $V$ of dimension $d$ defined over $F$:
$$
\# V(F) = q^d + O\left(q^{d - \frac{1}{2}} \right),
$$
where the implicit constant depends on the geometric complexity of the variety. Hence, in our situation, even if the relevant variety were shown to be geometrically irreducible of dimension $d =(n-1)^2$, the relevant constant may depend on $n$ and the geometric complexity of the varieties in this family. By contrast, the constant $42$ in Theorem \ref{mainthm1} is absolute and is independent of both $q$ and $n$, and even more, we obtain a better $O(1/q)$ relative error compared to the $O(1/\sqrt{q})$ relative error of Lang-Weil.  Consequently, the estimate in Theorem \ref{mainthm1} does not follow from the general Lang--Weil theorem; its stronger error term and uniformity arise from the particular linear-algebraic and conjugacy-class structure of the additive decomposition problem.

\section{Notation}\label{notation}
Throughout the article, $F$ denotes the finite field with $q$ elements, and $M_n(F)$ denotes the ring of $n \times n$ matrices with entries in $F$. Let $\id_n$ denote the $n \times n$ identity matrix; when there is no possibility of confusion, we simply denote it by $\id$, suppressing the size $n$.

Let $V := F^n$ be the $n$-dimensional vector space over $F$. The characteristic polynomial of $x \in M_n(F)$ is denoted by $\chi_x(T) \in F[T]$, that of a conjugacy class, say $X$, in $M_n(F)$ is denoted by $\chi_X(T) \in F[T]$, and that for $A \in \End_{F}(V)$ is denoted by $\chi_A(T)$.
 
Let $W$ be a vector subspace of $V$. Let $W[T] := W \otimes_{F} F[T]$, and we can think of $W[T]$ as $\Set{ w_0 + w_1 T + \ldots + w_d T^d : w_i \in W}$, a free $F[T]$-module of rank $m$. Let $S \in \End_{F}(W)$. By extension of scalars, $S$ acts $F[T]$-linearly on $W[T]$: $S_T := S \otimes 1$, so $S_T(w \otimes f(T)) := S(w) \otimes f(T)$. Instead of writing $S_T$, we suppress the subscript and simply write it as $S.$ For example, $T \id -S := T \id_{W[T]} - (S \otimes 1 ) \in \End_{F[T]}(W[T])$. Since $W$ is finite-dimensional, there is a natural identification $\End_{F}(W) \otimes_{F} F[T] \cong \End_{F[T]}(W[T]) \cong \End_{F}(W)[T]$. Its elements are polynomial expressions $A_0 + A_1 T + \ldots + A_d T^d$, where $A_i \in \End_{F}(W)$. In particular, we consider  $\adj(T \id-S)$, the adjoint of the linear transformation $T \id -S$,  as an element of $\End_{F}(W)[T]$. For $r \in W^*$ we consider the canonical $F[T]$-linear map $r \otimes 1 : W[T] \to F[T]$, and denote it simply by $r$.

We adopt the following \tit{nonstandard} notation. Let $Y$ be a conjugacy class  in $M_n(F)$ whose characteristic polynomial $\chi_Y(T) = (T- \lam) h(T)$, with $h(T) \in F[T]$ irreducible and degree $m := n-1$ and $h(\lam) \neq 0$. Let us denote by $L_y \sub V$ the $\lam$-eigenline (the one-dimensional subspace spanned by the  eigenvector (unique up to a scalar) corresponding to $\lambda$) fixed by $y$, and we fix  a complement $W$ of $L_y$ in $V$, so that $V = L_y \oplus W$ with $\dim W = m$. The matrix of $y$ is of the form 
$
\begin{pmatrix} \lam & r  \\ 0 & B \end{pmatrix},
$
where $r \in W^*, B \in \End(W)$. Note that we are not choosing a basis for $W$, as we wouldn't be dealing with coordinates explicitly, so the \tit{matrix} above is not standard, as its entries not elements of $F$.

Finally, let $\ol{F}$ be the algebraic closure of $F$. For $n \geq 0$, Let $\A_{F}^n$ be the $n$-dimensional affine space over $F$, and $\A_{\ol{F}}^n$ be the $n$-dimensional affine space over $\ol{F}$. Let $\P^n_F$ be the $n$-dimensional projective space over $F$ and $\P^n_{\ol{F}}$ denote the $n$-dimensional projective space over $\ol{F}$.
\section{Case $n=2$}

\subsection{Odd characteristic case}

Suppose the characteristic of $F$ is not $2$. 
Let $X$ be a conjugacy class whose characteristic polynomial is irreducible of degree $2$, and $Y$ be a conjugacy class in $M_2(F)$ whose characteristic polynomial is of the form $(T- \lam)(T - \mu)$ where $\lam, \mu \in F$ and $\lam \neq \mu$. Suppose $z$ is a nonscalar matrix having matching trace. Since every matrix in $M_2(F)$ can be written uniquely as a sum of a scalar matrix and a traceless matrix, i.e., $M_2(F) = F \cdot \mathrm{I}_2 \oplus \sl_2(q)$, the elements $x \in X, y \in Y$ and $z$ with traces, say $a,c$ and $a+c$ respectively, can be expressed as the following sums:
$$
x = \frac{a}{2} \mathrm{I}_2 + X_0, \; \;  y = \frac{c}{2} \mathrm{I}_2 + Y_0, \; \;  z = \frac{a+c}{2} \mathrm{I}_2+ Z_0,
$$
where the traces $\tr(X_0)$, $\tr(Y_0)$ and $\tr(Z_0)$ are all zero. Since $z$ is nonscalar, the matrix $Z_0$ is \tit{nonzero}. In the following, we don't use the Lie algebra structure of $\sl_2(q)$, but since it is standard to denote traceless matrices by $\sl_2(q)$, we adopt that notation here.

Consider the quadratic form $Q : \sl_2(q) \to F$ defined by $Q \begin{pmatrix} u & v \\ w & - u \end{pmatrix} := - \det \begin{pmatrix} u & v \\ w & - u \end{pmatrix} = u^2 + vw$ and let $\ip{A,B} := Q(A+B) - Q(A) - Q(B)$ be the associated bilinear form. We remark that since $Q$ is constant on conjugacy classes of $\sl_2(q)$,  we are justified in denoting $Q(X')$ for a conjugacy class $X'$ in $\sl_2(q)$ to be equal to $Q(A)$ for any $A \in X'$.

 Finally, consider  the quadratic character $\chi: F \to \Set{1, -1, 0}$ defined by $\chi(x) = 1$ if $x$ is a nonzero square in $F$, $-1$ if  $x$ is not a nonzero square in $F$, and $0$ if $x = 0$.

\begin{lem}\label{characterization_X_Y}
There exist nonzero $r,s \in F$ such that $r$ is a nonsquare and $s$ is a square, and $X$ is in bijection with the ($\GL_2(F)$-) \tit{conjugacy class} $X' := \Set{ A \in \sl_2(q) : Q(A) = r } $ in $\sl_2(q)$ and $Y$ is in bijection with the conjugacy class $Y' := \Set{ B \in \sl_2(q) : Q(B) = s}$ in $\sl_2(q)$.

\end{lem}
\begin{proof}
Let $C = \begin{pmatrix} a & b \\ c & d \end{pmatrix} \in M_2(F)$ and let $
\chi_C(T) := T^2 - (a+d) T + (ad-bc)$ be its characteristic polynomial. The discriminant of $\chi_C(T)$ is $\Delta(\chi_C) := (a+d)^2 - 4 (ad-bc) = (a-d)^2 + 4bc$.  We claim that
$$
\Delta(\chi_C) = 4 Q \left(C- \frac{(a+d)}{2} \id \right).
$$
Indeed,
$4 \cdot Q \left(C - \frac{(a+d)}{2} \id \right)= - 4 \cdot \det\left( C - \frac{(a+d)}{2} \id \right) = 4 \cdot \left( \left( \frac{a-d}{2} \right)^2 + bc \right) = (a-d)^2 + 4bc= \Delta(\chi_C)$,  as desired.

Now, we prove the characterization of $X'$. Let $r := Q(x- \frac{t}{2} \id)$, where $x \in X$ and $t = \tr(x)$. Since $\chi_x(T)$ is irreducible, its discriminant is a nonsquare. It follows from the above claim that $r$ is a nonsquare. For $A' \in \sl_2(q)$ such that $Q(A') = r$, its characteristic polynomial  $\chi_{A'}(T) = T^2 + \det(A') = T^2  - r$ is irreducible, because $r$ is not a square. Therefore, $\chi_{A'}(T)$ coincides with its minimal polynomial of $A'$, and $A'$ is conjugate to its rational canonical form $\begin{pmatrix} 0 & r \\ 1  & 0 \end{pmatrix}$. Thus, all matrices $A'$ in $\sl_2(q)$ with $Q(A') = r$ are conjugate to each other; in other words, $X'$ is a conjugacy class in $\sl_2(q)$ given by the above characterization. The map $x \mapsto  x- \frac{t}{2} \id$ is manifestly a  bijection of $X$ onto the set $X'$.

Now, we prove the characterization of $Y'$. Let $y \in Y$ and $s := Q(y- \frac{t}{2} \id)$ where $t = \tr(y)$. By the claim above, $s = \Delta(\chi_y)/4$. Since the  characteristic polynomial $\chi_y(T)$ of $y$ is a product of \tit{distinct} linear factors, its discriminant $\Delta(\chi_y)$ is a nonzero square, so $s$ is a nonzero square. On the other hand $\chi_{Y_0}(T) = \chi_{y - \frac{t}{2} \id}= T^2 - s$, so any matrix with characteristic polynomial $T^2 -s$ is conjugate to the diagonal matrix $\begin{pmatrix} \sqrt{s} & 0 \\ 0 & -\sqrt{s} \end{pmatrix}$, and the characterization of $Y'$ follows. Clearly, the map $y \mapsto \left(y- \frac{t}{2}I \right)$ is a bijection of $Y$ onto $Y'$.
\end{proof}

By the above Lemma \ref{characterization_X_Y}, it follows that the  map $(x,y) \mapsto (X_0, Y_0)$ is a bijection of  $\Set{(x,y) \in X \times Y : x + y = z}$ onto $\Set{(X_0, Y_0) \in \sl_2(q) \times \sl_2(q): Z_0 = X_0 + Y_0}$.  Therefore, we have the following characterization of the count $N_{X, Y}(z)$:
\begin{align*}
N_{X, Y}(z) &= \#\Set{ (x,y) \in X \times Y : x + y = z}, \\
&= \#\Set{(X_0, Y_0) \in X' \times Y' : Z_0 = X_0 + Y_0}, \\
&= \# \Set{Y_0 \in \sl_2(q) : Q(Y_0) = s, Q(Z_0 - Y_0) = r}, \\
&= \# \Set{Y_0 \in \sl_2(q): Q(Y_0) = s, \ip{Z_0,Y_0} = Q(Z_0) + s - r }.
\end{align*}
The third equality is a consequence of Lemma \ref{characterization_X_Y}, and the fourth equality is a consequence of the definition of the bilinear form. Now we further characterize $N_{X, Y}(z)$ in the following lemma.

\begin{lem}\label{further_char}
Let $Z_0 = \begin{pmatrix} \al & \beta  \\ \gamma  & - \alpha \end{pmatrix} \in \sl_2(q)$ and $Y_0 =  \begin{pmatrix} u & v \\ w & - u \end{pmatrix} \in \sl_2(q)$ be as above. Let $b := Q(Z_0) + s - r$. Then, 
$$
N_{X, Y}(z) = \# \Set{ (u,v,w) \in \A_{F}^3 :  u^2 + v w = s , 
2 \al u + \gam v + \beta w = b }.
$$
\end{lem}

\begin{proof}
Let $Y_0 = \begin{pmatrix} u & v \\ w & - u \end{pmatrix} \in \sl_2(q)$ such that $Q(Y_0) = s$. Then $Q(Y_0) = u^2 + v w$, and the equation $Q(Y_0) = s$ defines a quadric surface in the $3$-dimensional affine space $\A_{\overline{F}}^3$. 
On the other hand, the condition $\ip{Z_0, Y_0} = Q(Z_0 + Y_0) - Q(Z_0) - Q(Y_0) = b$ defines an affine plane in $\A_{\ol{F}}^3$. Indeed,
$
Z_0 + Y_0 = \begin{pmatrix} \al + u & \beta + v \\ \gam + w & - (\al + u) \end{pmatrix},
$
hence $Q(Z_0 + Y_0) = (\al+u)^2 + (\beta + v) (\gam + w)$ and 
\begin{align*}
\ip{Z_0, Y_0} &= Q(Z_0 + Y_0) - Q(Z_0) - Q(Y_0) \\
&= (\al + u)^2 + (\beta + v) (\gam + w) - (\al^2 + \beta \gam) - (u^2 + vw) \\
&= 2 \al u + \beta w + \gam v.
\end{align*}
\end{proof}

Let us continue the notation introduced above. Let $Z_0 = \begin{pmatrix} \al & \beta  \\ \gamma  & - \alpha \end{pmatrix} \in \sl_2(q)$ be as in Lemma \ref{further_char}. Since $Z_0$ is nonzero (here we are using that $z$ is nonscalar so that $Z_0$ is nonzero), at least one of $\al, \beta$ or $\gam$ is nonzero. Suppose $\beta$ is nonzero; the proof is similar in the cases where $\al \neq 0$ or $\gamma \neq 0$. Then we can solve the linear equation for $w$:
$$
w = \frac{b  - 2 \al u  - \gam v}{\beta}
$$
Substituting into the quadratic equation, we get
$$
u^2 + v \left( \frac{b- 2 \al u - \gam v}{ \beta} \right) = s
$$
so that we obtain a quadratic equation in two variables $u$ and $v$:
$$
\beta u^2  - 2 \al u v  - \gam v^2 + b v - s \beta = 0.
$$
The discriminant of this quadratic equation in $u$ is 
$$
\Delta(v) = (-2 \al v)^2 - 4 \beta (- \gam v^2 + b v - s \beta) = 4 ( (\al^2 + \beta \gam) v^2 - \beta b v + s \beta^2).
$$
Therefore the number of solutions $u \in F$ for each $v \in F$ is given by 
$$
1 + \chi( (\al^2 + \beta \gam) v^2  - \beta bv + s\beta^2)
$$
Therefore 
\begin{align}\label{N(z)exp}
N_{X, Y}(z) &= \sum_{v \in F} \left[ 1 + \chi( (\al^2 + \beta \gam) v^2  - \beta bv + s\beta^2) \right] \nonumber \\
&= q + \sum_{v \in F}  \chi( (\al^2 + \beta \gam) v^2  - \beta bv + s\beta^2)
\end{align}

Let $P(v) := (\al^2 + \beta \gam) v^2  - \beta bv + s\beta^2$, and first suppose that $Q(Z_0) = \al^2 + \beta \gam = 0$. Then $b$ which was defined to $Q(Z_0) + s -r$, in Lemma \ref{further_char} is equal to $s - r$, and $P(v) = - \beta(s-r)v+ s \beta^2$, which is a linear polynomial in $v$, linear because $\beta \neq 0$ and $s-r \neq 0$; recall $s$ is a square and $r$ is not a square. Therefore $v \mapsto P(v)$ is a bijection of $F$, so 
$$
\sum_{v \in F} \chi(P(v)) = \sum_{v \in F} \chi(v) = 0.
$$ 
Hence, it follows from Equation \ref{N(z)exp}, $N_{X, Y}(z) = q$ in this case. On the other hand,  suppose now that $Q(Z_0) = \al^2 + \beta \gam \neq 0$ so that $P(v)$ is a quadratic polynomial. The discriminant of $P(v)$ is given by  
\begin{align*}
\Delta(P(v)) &= (-\beta b)^2 - 4 ( \al^2 + \beta \gam) (s \beta^2) \\
& = \beta^2(b^2 - 4s Q(Z_0)).
\end{align*} 
 Since $\beta$ is nonzero by assumption, $\Delta(P(v)) = 0$ if and only if $b^2 = 4s Q(Z_0)$ equivalently $(Q(Z_0) + s - r)^2 = 4s Q(Z_0)$. Rearranging the terms we obtain the following equation:
$$
Q(Z_0)^2 - 2 (r + s) Q(Z_0) + (s-r)^2 = 0.
$$
But the discriminant of this equation is $4(r+s)^2 - 4 (s-r)^2 = 16rs$, which is not a square in $F$ since $r$ is not a square and $s$ is a square by Lemma \ref{characterization_X_Y}. Therefore, the equation has no solutions in $F$, i.e. for all nonzero $Z_0$ chosen as above we have $\Delta(P(v)) \neq 0$, whence $P(v)$ is not a square of a linear term. In this case, we have the following result. 

\begin{prop}\label{P(v)_expression}
Let $P(v) = av^2 + b v + c$, $a \neq 0$ be a polynomial in $v$ over $F$ with nonzero discriminant. Then 
$$
\sum_{v \in F} \chi(P(v)) = - \chi(a).
$$
In particular, $\abs{\sum_{v \in F} \chi(P(v))} \leq 1$.
\end{prop}
\begin{proof}
Let $P(v) = av^2 + b v + c$, $a \neq 0$. Completing the square, we have 
$$
P(v) = a \left( v + \frac{b}{2a} \right)^2 + \left( c - \frac{b^2}{4a} \right).
$$
Changing variables let $u := v + \frac{b}{2a}$, and $d = \left( c- \frac{b^2}{4a} \right)$, we get
$$
\sum_{v \in F} \chi(P(v)) = \sum_{u \in F} \chi(au^2 +d)
$$
Since the discriminant is nonzero, $d \neq 0$, so $\chi(au^2 +d) = \chi(a) \chi(u^2 + d/a)$. Letting $\mu = - d/a$,  it suffices to prove that 
$$
\sum_{u \in F} \chi(u^2 - \mu) = -1.
$$

Let $N = \# \Set{(u,t) \in F^2 : t^2 = u^2 - \mu}$. As before, for each fixed $u$, the equation $t^2 = u^2 - \mu$ has $1 + \chi(u^2 - \mu)$ solutions in $t$. Therefore 
$$
N = \sum_{u \in F} (1 + \chi(u^2 - \mu)) = q + \sum_{u \in F} \chi(u^2 - \mu).
$$

On the other hand, let us compute $N$ in another way. Rewrite the defining equation of $N$ above as $t^2 = u^2 - \mu$ as $u^2 - t^2 = \mu$ so that by changing variables $A := u-t$ and $B := u+t$, we have an equation $AB = \mu$. Since $\mu \neq 0$, we must have $A \neq 0$, so that given $A \neq 0$,  $B$ is uniquely determined, whence there are exactly $q-1$ solutions, hence $N = q-1$. Substituting this value for the expression of $N$ above, we obtain the desired result.
\end{proof}

We now prove Theorem \ref{main_theorem_n=2} in the odd characteristic case.

\begin{proof}
The result follows from  Equation \eqref{N(z)exp} and Proposition \ref{P(v)_expression} if $z$ is not a scalar matrix.
\end{proof}

\subsection{Characteristic $2$ case}

Let $q$ be even. The decomposition into sum of a scalar matrix and a traceless matrix  in the odd characteristic case does not apply here. So we work with trace-determinant equations directly. Recall the notation. Let $X$ be a conjugacy class in $M_2(F)$ whose characteristic polynomial $\chi_X(T) := T^2 + \al T + \beta \in F[T]$ is irreducible. Clearly, $\alpha \neq 0$, otherwise $\chi_X(T)$ is reducible, since every element in $F$ is a square. Let $Y$ be a conjugacy class whose characteristic polynomial $\chi_Y(T) := T^2+ \gamma T + \delta$ is a product of distinct linear factors. Since $\chi_Y(T)$ splits into distinct factors it follows that $\gamma$ is nonzero too.

Let $z \in M_2(F)$ be a nonscalar matrix with trace $\al + \gam $. Then 
\begin{align*}
N_{X, Y}(z) &= \# \Set{(x,y) \in X \times Y : x + y = z} \\
&= \# \Set{y \in Y : z + y \in X} \\
&= \# \Set{y \in M_2(F) : \tr(y) = \gam, \; \det(y) = \delta, \; \tr(z+y) = \alpha, \; \det(z + y) = \beta}\\
&= \# \Set{y \in M_2(F) : \tr(y) = \gam, \; \det(y) = \delta, \; \det(z+y) = \beta}.
\end{align*}
The last equality follows from the fact that the trace condition $\tr(z +y) = \alpha$ on $X$ follows from the condition $\tr(y) = \gam$: Indeed, since $\tr(z) = \al + \gam $ we have $\tr(z+y) = \tr(z) + \tr(y) = (\al + \gam) + \gam = \al$.

Now let us consider the first condition. Let $E = \Set{y \in M_2(F) : \tr(y) = \gam}$, and $y = \begin{pmatrix} u & v \\ w & d \end{pmatrix} \in E$. Then $u + d = \gam$ so that $d = \gam + u$, and $y$ is  of the form $y = \begin{pmatrix} u & v \\ w & u + \gam \end{pmatrix}$. Therefore  we may identify $E$ with the affine space $\A_{F}^3$ with the coordinates $(u,v,w)$ 

Now consider the affine quadric surface $S$ defined by the condition $\det(y) = \delta$; clearly the condition is equivalent to the polynomial condition $u^2 + \gam u + v w + \delta = 0$ in the affine space $\A^3$ (identified with $E$). Thus the conjugacy class $Y$ corresponds to the surface $S$ in $\A^3_{F}$ defined by $u^2 + \gam u + v w + \delta = 0$.

Now consider the affine plane $\Pi_z$ defined by the condition $\det(z+y) = \beta$. Let $z := \begin{pmatrix} p & r \\ s & p + \theta \end{pmatrix}$, where $\theta = \tr(z) = \al + \gam$. As in the odd characteristic case, consider the quadratic form $Q : M_2(F) \to F$ defined by $Q(a) := \det(a)$, and let 
$$
B(a,h) = \det(a+h) + \det(a) + \det(h)
$$
be the associated bilinear form.
If $a = \begin{pmatrix} a_1 & a_2 \\ a_3 & a_4 \end{pmatrix}$ and $h = \begin{pmatrix} h_1 & h_2 \\ h_3 & h_4 \end{pmatrix}$ then 
\begin{equation}\label{B_equation}
B(a,h) = a_1 h_4 + a_4 h_1 + a_2 h_3 +a_3 h_2.
\end{equation} 
In our case, since $\det(z+y) = \det(z) + \det(y) + B(z,y)$ the condition $\det(z+y) = \beta$ is equivalent to $B(z,y) = \beta + \det(z) + \delta$; recall from the previous paragraph that $\det(y) = \delta$. Let 
$
k_0 := \beta + \det(z) + \delta.
$
Thus, it follows from Equation \eqref{B_equation} that $B(z,y) = p(u + \gam) + (p + \theta) u + r w + s v = \theta u + sv + rw + p \gam$. Therefore the condition $B(z,y) = k_0$ is equivalent to the condition 
$$
\theta u + s v + r w + p \gam + k_0 = 0.
$$
Let $k = p \gam + k_0$. Then consider the intersection  $C_{\trm{aff}} = S \cap \Pi_z$ in $A_{\ol{F}}^3$. It is given by 
$$
C_{\trm{aff}} = \Set{(u,v,w) \in \A^3_{\ol{\F}_q} : \begin{array}{l} u^2 + \gam u + v w + \delta = 0, \\ \theta u + s v + r w + k = 0 \end{array}}
$$
By construction, $N_{X, Y}(z) = \#C_{\trm{aff}}(F)$, the $F$-rational points of the intersection which is a conic $C_{\aff}$. Now we show that $C_{\aff}$ is a nondegenerate conic, by viewing inside the projective space. It is well known that any nondegenerate \tit{projective} conic has exactly $q+1$ rational points. The line at infinity meets this conic in at most $2$ points. Therefore, its affine restriction $C_{\aff}$ has at least $q-1$ points and at most $q+1$ points, and the result follows.

\vspace{3mm}
Let $\P^3_{\ol{F}}$ be the projective space with $\A^3_{\ol{F}}$ embedded as $(u,v,w) \mapsto [u:v:w:1]$. Homogenezing the defining equation of the quadratic surface $S$, the  projective closure $\ol{S}$ of $S$ in $\P^3_{\ol{F}}$ is given by 
$$
f(u,v,w,t) = u^2 + \gam u t + v w  + \delta t^2 = 0.
$$
Similarly homogenizing the definition equation of the affine plane $\Pi_z$ we get the corresponding projective plane in $\P^3_{\ol{F}}$ defined by the equation
$$
l(u,v,w,t) = \theta u + s v + r w + kt = 0.
$$
Then the projective closure of the affine conic $C_z$ is  $\ol{C_z} := \ol{S} \cap \ol{\Pi_z} \subset \P^3_{\ol{F}}$, and we can recover the affine conic $C_{\trm{aff}}$ as $C_{\trm{aff}} = \ol{C_z} \cap \Set{ t \neq 0}$.

We first prove that $\ol{S}$ is smooth. We have $\frac{\partial f}{\partial u} = \gam t$, $\frac{\partial f}{\partial v} = w$, and $\frac{\partial f}{\partial w} = v$ and $\frac{\partial f}{\partial t} = \gam u$. Because $\gam \neq 0$ (being the trace of the conjugacy class $Y$ whose characteristic polynomial is a product of distinct linear factors), it follows that the partial derivatives vanish only if $t = 0, w = 0, v= 0, u = 0$. Hence $\ol{S}$ is smooth.

The curve $\ol{C_z} = \ol{S} \cap \ol{\Pi_z}$ is a plane conic inside the projective plane $\ol{\Pi_z}$; it is nonempty because any quadratic form in at least three variables over $F$ has a nontrivial zero. Recall that a projective plane conic is degenerate if and only if it is singular. Since $\ol{S}$ is smooth, the conic $\ol{C_z}$ is singular if and only if $\ol{\Pi_z}$ is tangent to $\ol{S}$ at some point. This means there is a point $P = [u:v:w :t] \in \P^3$ such that $f(P) = 0$ and $l(P) = 0$ and the gradient $\nabla f$  of $f$ is proportional to the gradient $\nabla l$ of $l$. Since $\nabla f = (\gam t, w, v, \gam u)$ and $\nabla l = (\theta, s, r, k)$, it follows that the point $P$ is singular if there exists a $\lam \in \ol{F}^{\times}$ such that $\gam t = \lam \theta, w = \lam s, v = \lam r, \gam u = \lam k$. Since $\gam$ is nonzero, we have
$$
[u:v : w : t] = [\gam^{-1} \lam k : \lam r : \lam s : \gam^{-1} \lam \theta ] = [k:\gam r, \gam s : \theta]
$$
Clearly $l(k,\gam r, \gam s, \theta) = \theta k + s(\gam r) + r (\gam s) + k \theta = 0$, so it remains to verify that this point is on $\ol{S}$ to prove that the point is singular.  Substituting into $F$, we get
$$
f(k, \gam r, \gam s, \theta) = k^2 + \gam k \theta + \gam^2 r s + \delta \theta^2
$$
so that the point $[u:v:w:t]$ is singular if and only if $k^2 + \gam \theta k + \gam^2 r s + \delta \theta^2 = 0$. We now show that this never happens. First suppose that $\theta \neq 0$, so that the point is a finite point (because $\gam t = \lam \theta \implies t = \gam^{-1} \lam \theta$). Then it corresponds to the affine point $(u,v,w) = \left(\frac{k}{\theta}, \frac{\gam r} {\theta}, \frac{\gam s} { \theta} \right)$ in $C_{\aff}$. At a singular point of the affine conic, the gradients of the two affine equations 
$$
u^2 + \gam u + v w + \delta = 0
$$
and
$$
\theta u + s v + r w + k = 0
$$
are proportional. The gradient of the quadric equation is $(\gam, w, v)$ and that of the plane equation is $(\theta, s, r)$. Thus proportionality means $(\gam, w, v)$ and $(\theta, s, r)$ are proportional. This is equivalent to the linear forms $h \mapsto B(y,h)$ and $h \mapsto B(z,h)$ are proportional on $\sl_2(q)$. Indeed for $h = \begin{pmatrix} h_1 & h_2 \\ h_3 & h_1 \end{pmatrix} \in \sl_2(q)$, we have $B(y,h) = \gam h_1 + w h_2 + v h_3$	 and $B(z, h) = \theta h_1 + s h_2 + r h_3$, so that the proportionality of gradients is equivalent to the proportionality of the linear forms.  Hence there exists $\rho \in F$ such that $B(z,h) = \rho B(y,h)$ for all $h \in \sl_2(q)$. Equivalently, $B(z + \rho y, h) = 0$ for all trace-zero $h$.

\begin{lem}
$B(a,h) = 0$ for all trace-zero $h$ if and only if $a$ is a scalar.
\end{lem}
\begin{proof}
Indeed, if $a = \begin{pmatrix} a_1 & a_2 \\ a_3 & a_4 \end{pmatrix}$ then $B(a,h) = (a_1 + a_4)h_1 + a_2 h_3 + a_3 h_2$. If $B(a,h) = 0$ for all $h$, then $a_1 + a_4 = 0, a_2, a_3 =0$ so that $a_1 = a_4$. Thus $a$ is a scalar. The converse is clear. 
\end{proof}

Therefore $z + \rho y = \mu I$ for some $\mu \in F$.  Hence $z + y = \mu I + (\rho + 1) y$. But $y \in Y$ is diagonalizable over $F$ with distinct eigenvalues. Therefore, if $\rho + 1 = 0$, then $z + y$ is a scalar, and if $\rho +1 \neq 0$ then $z + y$ is split over $F$. In either case, $z+y$ cannot have irreducible characteristic polynomial, contrary to the hypothesis that  $z + y \in X$ and $X$ has irreducible characteristic polynomial.
Thus $\ol{C_z}$ cannot be singular when $\theta \neq 0$.

Now assume that $\theta = 0$, so that the presumably singular point is a point at infinity. Since $\theta = \al + \gam$ we have $\al = \gam$, and the possible singular point now is $[k: \gam r : \gam s : 0]$. As before, if it is singular, it must be a point of $\ol{S}$ so that substituting in the defining equation of $\ol{S}$ we get
$$
k^2 + \gam^2 rs = 0
$$
Equivalently, $(k/\gam)^2 = rs$. Set $\eta = \frac{k}{\gam}$ so that $\eta^2 = rs$. Let $\mu := p  + \eta$. Since $\theta = 0$, we have $z = \begin{pmatrix} p & r  \\ s & p \end{pmatrix}$ so that $\det(z) = p^2 + rs = p^2 + \eta^2 = (p + \eta)^2  = \mu^2$. Now, recall $k = p \gam + k_0$. Since $k = \gam \eta$, this gives 
$
k_0 = p \gam + k = \gam (p + \eta) = \gam \mu.
$
But $k_0 = \beta + \det(z) + \delta$, therefore $\gam \mu = \beta + \mu^2 + \delta$. So 
$$
\mu^2 + \gam \mu + \beta + \delta = 0 \implies \left( \frac{ \mu}{ \gam } \right)^2 + \left( \frac{\mu}{\gam} \right) + \frac{\beta + \delta}{\gam^2} = 0.
$$
Hence $\frac{\beta + \delta}{\gam^2}$ lies in the image of the Artin-Schreier map $a \mapsto a^2 + a$. Over $F$, the image of this map is exactly the elements with absolute trace $0$:
$$
\trm{Im}(a \mapsto a^2 + a) = \Set{ b \in F : \tr_{F/\F_2}(b) = 0 }
$$
Therefore 
$$
\Tr_{F/\F_2} \left( \frac{\beta + \delta} {\gam^2} \right) = 0.
$$
But $\al = \gam$, so 
$
\frac{\beta + \delta}{\gam^2} = \frac{\beta}{\al^2} + \frac{\delta}{\gam^2}.
$
Now recall that for $a \neq 0$ the polynomial $t^2 + at + b$ has a root in $F$ if and only if $\tr_{F/\F_2} \left( \frac{b}{a^2} \right) = 0$. Therefore $t^2 + \al t + \beta$ is irreducible implies $\tr_{F/\F_2} \left( \frac{\beta}{\al^2} \right) = 1$, and since $t^2 + \gam t + \delta$ split with distinct roots, we have $\tr_{F/\F_2} \left( \frac{\delta}{\gam^2} \right) =0$. Thus 
$$
\tr_{F/\F_2} \left( \frac{\beta + \delta}{\gam^2} \right) = 1 + 0 = 1
$$
This contradicts the conclusion that the trace is $0$. Therefore $\ol{C_z}$ cannot be singular even when $\theta = 0$. 

We have shown that $\ol{C_z}$ is nondegenerate projective conic over $F$.  Since the conic has an $F$-rational point, and since smooth projective conic over $F$ with an $F$-rational point is $F$-isomorphic to the projective line $\P^1_{\ol{F}}$, the curve $\ol{C_z}$ has exactly $(q+1)$ $F$-rational points. Thus $\# \ol{C_z}(F) = q + 1$, so that its affine restriction is obtained by removing the points at infinity:
$$
C_{\aff} = \ol{C_z} \setminus ( \ol{C_z} \cap \Set{t = 0})
$$
The hyperplane $t = 0$ intersects the plane $\ol{\Pi_z}$ in a line, and a line meets a nondegenerate conic in at most $2$ points. Hence 
$
0 \leq \#(\ol{C_z} \cap \Set{t = 0})(F) \leq 2.
$
Therefore
$$
q + 1 - 2 \leq \#C_{\aff}(F) \leq q+1
$$
so that $N_{X, Y}(z)= \#C_{\aff}(F)$ satisfies the desired inequality $\abs{N_{X, Y}(z) - q} \leq 1$. This completes the proof of Theorem \ref{main_theorem_n=2}.

\section{General Case}

Let  $n \geq 2$. Let $X$ be a conjugacy class in $M_n(F)$ whose characteristic polynomial is irreducible of degree $n$, and $Y$ be a conjugacy class in $M_n(F)$ whose characteristic polynomial is of the form $(T - \lambda)h(T)$, where $\lambda \in F$ and $h(T)$ is an irreducible polynomial over $F$ of degree $n-1$ and $h(\lambda) \neq 0$ ; so when $n = 2$ the linear factors are distinct. Recall the notation \S $2$.

\begin{lem}\label{conjugacy_characteristic_equivalence}
For $x,y \in M_n(F)$, we have
\begin{align*}
x \in X & \trm{ if and only if } \chi_x(T) = \chi_X(T), \\
y \in Y & \trm{ if and only if } \chi_y(T) = \chi_Y(T).
\end{align*}
Furthermore, under the natural action of $M_n(F)$ on the vector space $V$,  $x \in X$ acts irreducibly on $V$, i.e. there is no proper nonzero $x$-invariant subspace of $V$.
\end{lem}
\begin{proof}
For both equivalences, the forward direction is clear. For the other direction, let $x' \in M_n(F)$ with characteristic polynomial equal to $\chi_X(T)$. Since $\chi_X(T)$ is irreducible, the minimal polynomial of $x'$ is also $\chi_X(T)$. By the theory of rational canonical forms, it follows that $x'$ is conjugate to the companion matrix of $\chi_X(T)$. Likewise, any $x \in X$ is conjugate to the companion matrix of $\chi_X(T)$. It follows that $x$ and $x'$ are conjugate, hence $x'$ belongs to $X$.

Let $y' \in M_n(F)$ with characteristic polynomial $\chi_{y'}(T) = \chi_Y(T)$. Since $\chi_Y(T)$ is a product of relatively prime irreducible factors, it follows that the minimal polynomial of $y'$ is also equal to $\chi_Y(T)$, so that $y'$ is conjugate to the companion matrix of $\chi_Y(T)$. As above, any $y \in Y$ is conjugate to companion matrix of $\chi_Y(T)$, and therefore $y$ and $y'$ are conjugate.

For the final assertion, we simply note that if $W$ is a nonzero proper $x$- invariant subspace of $V$ then the characteristic polynomial $\chi_x(T) = \chi_X(T)$ is the product of nonconstant polynomials $\chi_W(T)$ and $\chi_{V/W}(T)$ contradicting to the irreducibility of $\chi_X(T)$.
\end{proof}
Let $z$ be a nonscalar matrix in $M_n(F)$ with matching trace. We have $\chi_Y(T) = (T- \lam) h(T)$, with $h(T) \in F[T]$ irreducible and degree $m := n-1$ and $h(\lam) \neq 0$. See notation \S $2$. Let us denote by $L_y \sub V$ the $\lam$-eigenline fixed by $y$, and $W$ be a fixed complement of $L_y$ in $V$, so that $V = L_y \oplus W$ with $\dim W = m$. The matrix of $y$ is of the form 
$
\begin{pmatrix} \lam & r  \\ 0 & B \end{pmatrix},
$
where $r \in W^*, B \in \End(W)$; see notation \S $2$. Thus the characteristic polynomial of $y$ is $\chi_y(T) = \chi_Y(T) = (T- \lam) \chi_B(T)$, and it follows from the equivalence in Lemma \ref{conjugacy_characteristic_equivalence} that $y \in Y$ if and only if $\chi_B(T) = h(T)$. Let 
$$
\CalC_h = \Set{B \in \End_{F}(W) : \chi_B(T) = h(T)}
$$
Since $h$ is irreducible of degree $m$, again by Lemma \ref{conjugacy_characteristic_equivalence}, it follows that  $\CalC_h$ is a single $GL(W)$-conjugacy class in $\End_{F}(W)$.

\begin{lem}\label{charY}
Fix $y \in Y$. Then, for $y' \in M_n(F)$,  
$$
y' \in Y \trm{ such that } L_{y'} = L_y \trm{ if and only if } y' = \begin{pmatrix} \lam & r \\ 0 & B \end{pmatrix} \trm{ for some } r \in W^*, B \in \CalC_h.
$$
\end{lem}
\begin{proof}
Based on the discussion preceding this lemma, it remains to prove that if we choose $B \in C_h$ and $r \in W^*$ and define $y' := \begin{pmatrix} \lam & r \\ 0 & B \end{pmatrix}$ then $y' \in Y$ such that $L_{y'} = L_y$. First, because $B \in \CalC_h$, we have $\chi_B(T)= h(T)$ so that $\chi_{y'}(T) = (T- \lam) \chi_B(T) = (T - \lam) h(T) = \chi_Y(T)$, so that by Lemma \ref{conjugacy_characteristic_equivalence} it follows that $y' \in Y$. It remains to check that the $\lam$-eigenline $L_{y'}$ of $y'$ is $L_y$. From the decomposition $V = L_y \oplus W$, let $L_y = \ip{e}$, and $\al e + w \in V$ be an eigenvector of $y'$, where $w \in W$. By block multiplication, 
$$	
y'(\al e + w	) = \begin{pmatrix} \lam & r \\ 0 & B \end{pmatrix} \begin{pmatrix} \al e \\ w \end{pmatrix} = \begin{pmatrix} \lam \al e + r(w)e \\ B(w) \end{pmatrix}
$$
it follows that 
$$
y'(\al e + w ) = \lam \al e + r(w) e + B(w)
$$
so that 
$$
(y' - \lam I)(\al e + w) = r(w) e + (B - \lam I)w.
$$
Since $\al e + w $ is an eigenvector of $y'$ we have $r(w) = 0$ and $(B- \lam I) w = 0$. But $B$ has characteristic polynomial $h(T)$ and by assumption $h(\lam) \neq 0$, so  $\lam$ is not an eigenvalue of $B$. Therefore $B - \lam I$ is invertible, so that $w = 0$, hence $r(w) = 0$. It follows that $L_{y'} = \Ker(y' - \lam I) = \Set{ \al e : \al \in F} = L_y$ as desired.
\end{proof}
Recall the following standard result: 
\begin{lem}\label{schur}
Let $R$ be a commutative ring, let $M \in M_m(R)$, let $\alpha \in R$ be a scalar, $u \in M_{1,m}(R)$, and $v \in M_{m,1}(R)$. Then 
$$
\det \begin{pmatrix} \al & u \\ v & M \end{pmatrix} = \al \det(M) - u \;  \adj(M) \; v.
$$
\end{lem}
\begin{proof}
This is the standard identity and follows directly by cofactor expansion along the first row.
\end{proof}

\noindent Let $y \in Y$. As above let $W$ be a complement of $L_y$ in $V$.  By Lemma \ref{charY},  $y = \begin{pmatrix} \lam & r \\ 0 & B \end{pmatrix}$ for some $r \in W^*$ and $B \in \End(W)$. Let $z \in M_n(F)$. Then $z$ is of the form 
$
z := \begin{pmatrix} a & b \\ c & A \end{pmatrix}
$
where $a \in F, b \in W^*, c \in W$, and $A \in \End(W)$. Let $x = z - y = \begin{pmatrix} a  - \lam & b -r \\ c & A- B \end{pmatrix}$. With $S := A- B$, the  characteristic polynomial of $x$ is given by, according to Lemma \ref{schur}, 
\begin{align*}
\chi_x(T) &= \det(TI-x) = \det \begin{pmatrix} T - a + \lam & - (b-r) \\ -c & TI - S \end{pmatrix} \\
&= (T- a + \lam) \det(TI-S) - (b-r) \adj(TI-S)c.
\end{align*}
 Recall the notation \S $2$. For $c \in W$, let 
$$
U_B(T) := \adj(TI-S)c
$$
so that $U_B(T) \in W[T]$. Let $Q_B(T) := \det(TI-S)$. Then $\chi_x(T) = (T - a + \lam)Q_B(T) - b U_B(T) + r U_B(T)$. Therefore the condition $\chi_x(T) = \chi_X(T)$ is equivalent to 
$$
rU_B(T) = \chi_X(T) - (T - a + \lam) Q_B(T)  + b U_B(T).
$$
Define
$$
R_B(T) := \chi_X(T)  - (T - a + \lam) Q_B(T) + b U_B(T).
$$
Then the equation becomes $rU_B(T) = R_B(T)$. We now claim that $R_B(T)$ is a polynomial of degree at most $m-1$.

\begin{lem}\label{R_B(T)_degree}
The degree of $R_B(T)$ is at most $m-1$ (= $n-2$).
\end{lem}
\begin{proof}
Since the degree of $\chi_X(T)$ is $n = m +1$ and the degree of $(T- a + \lam)Q_B(T)$ is also $n = m+1$, and since both are monic polynomials,  the leading terms cancel. We must check the coefficient of $T^m$.  Let $\chi_X(T) = T^{m+1} - \tr(X) T^m + \ldots$. Since $Q_B(T) = \det(TI- S) = \det(TI- (A-B))$ we have $Q_B(T) = T^m  - \tr(S) T^{m-1}+ \ldots$. Therefore $(T - a + \lam) Q_B(T)$ has $T^m$-coefficient $- \tr(S) - (a- \lam)$. But $S = A - B$, so $\tr(S)= \tr(A)- \tr(B)$. Also, $\tr(z) = a + \tr(A)$ and $\tr(Y) = \lam + \tr(B)$. Thus
$$
(a- \lam	) + \tr(S) = a - \lam + \tr(A)- \tr(B) = \tr(z) - \tr(Y).
$$
By hypothesis, $\tr(z) = \tr(X) + \tr(Y)$, so $(a- \lam) + \tr(S) = \tr(X)$. Therefore the coefficient of $T^m$ in $(T- a + \lam) Q_B(T)$ is $- \tr(X)$ which is exactly the coefficient of $T^m$ in $\chi_X(T)$. So the degree $n = m+1$ and degree $m$ terms cancel in $\chi_X(T) - (T - a + \lam) Q_B(T)$. Since $b U_B(T)$ has degree at most $m-1$, it follows that $R_B(T) \in F[T]_{\leq m-1}$.
\end{proof}
Thus the equation $r U_B(T) = R_B(T)$ is an equation inside the $m$-dimensional $F$-vector space $F[T]_{\leq m-1}$, the vector space of polynomials of degree at most $m-1$ over $F$.

\begin{lem}\label{cyclic_vector_lemma}
Let $W$ be a subspace of $V$ of dimension $m$, and $S \in \End(W)$, and $c \in W$. Define
$$
\Phi_{S,c} : W^* \to F[T]_{\leq m-1}
$$
by $\Phi_{S,c}(r) = r(\adj(TI-S)c)$; since the degree of $\adj(TI-S)c$ is at most $m-1$, this map is well-defined.
Then $\Phi_{S,c}$ is an isomorphism if and only if $c$ is a cyclic vector for $S$, i.e. $\Set{c, Sc, S^2 c, \ldots, S^{m-1}c}$ forms a basis of $W$.
\end{lem}
\begin{proof}
Let 
$$
Q(T) := \det(T I -S) = T^m +q_1 T^{m-1} + q_2 T^{m-2} + \ldots + q_m.
$$
Recall the standard identity 
\begin{align*}
\trm{adj}(TI-S) &= T^{m-1}I + T^{m-2}(S + q_1 I) + T^{m-3}(S^2 + q_1 S + q_2I) + \ldots + \\
& (S^{m-1} + q_1 S^{m-2} + \ldots + q_{m-1}I).
\end{align*}
Then 
$$
\trm{adj}(TI-S)c = T^{m-1}v_0 + T^{m-2}v_1 + \ldots + v_{m-1},
$$
where $v_0 = c$, $v_1 = Sc + q_1 c$, $v_2 = S^2 c + q_1 Sc + q_2 c$, and in general 
$$
v_j = S^j c + \trm{ a  linear combination of } c, Sc, \ldots, S^{j-1}c.
$$
Since $v_0, v_1,\ldots, v_{m-1}$ is obtained from the list $c, Sc, \ldots, S^{m-1}c$ by an upper triangular change of basis with diagonal entries $1$, it follows that $v_0, \ldots, v_{m-1}$ form a basis of $W$ if and only if $c, Sc, \ldots, S^{m-1}c$ form a basis of $W$.

Now 
\begin{align*}
\Phi_{S,c}(r) &= r ( \adj(TI-S)c ) = r \left( T^{m-1}v_0 + \ldots + v_{m-1} \right) \\
&= r(v_0) T^{m-1} + r(v_1) T^{m-2} + \ldots + r(v_{m-1}).
\end{align*}
The map 
$$
r \mapsto (r(v_0), \ldots, r(v_{m-1}))
$$
is an isomorphism of $W^*$ to $F^m$ if and only if $v_0, \ldots, v_{m-1}$ form a basis of $W$. Indeed, if the map is an isomorphism and $\sum_{i = 0}^{m-1} a_i v_i = 0$ then choosing $r_i$ such that $r_i(v_j) = \delta_{ij}$ we see that each $a_i = 0$ , hence $v_0, \ldots, v_{m-1}$ are linearly independent, and since they are $m$ in number, they form a basis of $W$. Conversely, if $v_0, \ldots, v_{m-1}$ form a basis of $W$, then choosing a basis of $W^*$ dual to $v_0, \ldots, v_{m-1}$ we see that the standard basis vectors of $F^m$ are contained in the image, so that the map is an isomorphism. The result follows.
\end{proof}
Now we come to the crucial result of this section. For $B \in \CalC_h$ and $r \in W^*$, set $y_{B,r} = \begin{pmatrix} \lam & r \\ 0 & B \end{pmatrix}$. By Lemma \ref{charY}, $y_{B,r} \in Y$, and its $\lam$-eigenline is $L$.
\begin{prop}\label{cyclic_noncyclic}
Let $z$ be as above. For every $B \in \CalC_h$, the following holds:
$$
\#\Set{r \in W^* : z - y_{B,r} \in X} = \begin{cases} 1, & \trm{ if } c \trm{ is cyclic for } A-B, \\
0, & \trm{ if } c \trm{ is not cyclic for } A - B.
\end{cases}
$$
\end{prop}
\begin{proof}

Fix $B \in \CalC_h$, and put $S := A -B$. We consider separately the cases in which $c$ is cyclic and noncyclic for $S$.  Suppose that $c$ is not cyclic for $S$. We claim that $z- y_{B,r} \not \in X$ for every $r \in W^*$.
 Let $x := \begin{pmatrix} a- \lam & b - r \\ c & S \end{pmatrix}$, where $S = A  - B$. Suppose $c$ is not cyclic for $S$. Define 
$$
U = \ip{c, Sc, S^2c, \ldots, S^{m-1}c} \sub W.
$$
Then $U$ is a proper $S$-stable subspace of $W$. Let $L := L_y = Fe$ and let $E = L \oplus U$, which is a nonzero proper subspace of $V$. We claim that $E$ is $x$-stable. Indeed, let $\al e+ u  \in  V$ with $\al \in F, u \in U$. Then, from the above matrix representation of $x$	, we see that  
$$
x(\al e + u) = ((a- \lam) \al + (b-r)(u))e + \al c + Su.
$$
Since $u \in U$ and $U$ is $S$-stable, $\al c + S u \in U$. Therefore $x (L \oplus U) \sub L \oplus U$, hence $E$ is a nonzero proper $x$-stable $F$-subspace of $V$, contrary to the fact that $x$ acts irreducibly on $V$; refer to the last statement of Lemma \ref{conjugacy_characteristic_equivalence}.

Suppose now that $c$ is cyclic for $S = A - B$. 
Since $c$ is cyclic for $A-B$, it follows from Lemma \ref{cyclic_vector_lemma} that the linear map $\Phi_B : W^* \to F[T]_{\leq m-1}$ given by $\Phi_B(r) = rU_B(T)$ is an isomorphism. The condition $x = z - y \in X$ is equivalent to $\chi_{z-y} (T) = \chi_X(T)$ which in turn is equivalent to $\Phi_B(r) = R_B(T)$. Indeed if $z - y \in X$ then by Lemma \ref{conjugacy_characteristic_equivalence}, we have $\chi_{z-y}(T) = \chi_X(T)$. On the other hand, by the computation before Lemma \ref{R_B(T)_degree}, it follows that the condition $\chi_{z-y}(T) = \chi_X(T)$ is equivalent to $r(U_B(T)) = R_B(T)$. But $r(U_B(T)) = \Phi_B(r)$, so if $\Phi_B$ is an isomorphism, then there exists a \tit{unique} $r$ such that $r(U_B(T)) = R_B(T)$, equivalently $z-y \in X$. If $\Phi_B$ is not an isomorphism, then $c$ is not a cyclic vector for $A-B$, hence $z-y \not \in X$, equivalently there exists no $r \in W^*$ such that $r(U_B(T)) = R_B(T)$, hence $\chi_{z-y}(T) \neq \chi_X(T)$, so that $z-y \not \in X$ by Lemma \ref{conjugacy_characteristic_equivalence}. 
Hence there exists exactly one $r \in W^*$ such that $z-y_{B,r} \in X$  when $c$ is cyclic for $A-B$, and there is no such $r$ when $c$ is noncyclic.
 
\end{proof}

\subsection{Uniform size of $\CalC_h$}

\begin{lem}\label{m_conj_count}
Let $Z \subset M_m(F)$ be a $\GL_m(F)$-conjugacy class whose characteristic polynomial is irreducible of degree $m$. Then
$$
\abs{\#Z - q^{m^2 - m}} \leq 2 q^{m^2 - m -1}.
$$
\end{lem}
\begin{proof}
For $B \in Z$, the centralizer $C_{\GL_m(q)}$ of $B$ in $\GL_m(q)$ is isomorphic to the multiplicative group of the finite field with $q^m$ elements, so that its cardinality is $q^m-1$. Therefore 
$$
\#Z = \frac{\# \GL_m(q)}{q^m-1}= \frac{q^{m^2} \prod_{i = 1}^m (1- q^{-i})}{q^m-1} = q^{m^2 - m} \prod_{i =1}^{m-1} (1 - q^{-i})
$$
Now consider the product $P_m(q) := \prod_{i = 1}^{m-1} (1 - q^{-i})$. Since each $0 < 1- q^{-i} \leq 1$ we have $0 < P_m(q) \leq 1$. Let $a_i = q^{-i}$. Then from the inequality 
$$
1  - \prod_{i = 1}^{m-1} (1 - a_i) \leq \sum_{i = 1}^{m-1} a_i
$$
we get
$$
1- P_m(q) \leq \sum_{i = 1}^{m-1} q^{-i} \leq \frac{1}{q-1} \leq \frac{2}{q}.
$$
Thus 
$$\abs{\#Z - q^{m^2 - m}} = \abs{q^{m^2 -m}  P_m(q) - q^{m^2 - m} } \leq q^{m^2 -m} \abs{ 1 - P_m(q)} \leq q^{m^2-m} \cdot \frac{2}{q} = 2 q^{m^2-m-1}
$$
as desired.
\end{proof}

\subsection{A fixed-subspace counting lemma}

We use the following result of Ram \cite[Lemma $2.7$]{Ra} in the following lemma. A map $T : U \to W$ is said to be \tit{simple} if it has no nonzero $T$-invariant subspace contained in $U$. 
\begin{lem}\label{ram}
Let $f \in F[T]$ be a monic irreducible polynomial of degree $n$ and suppose $W \sub V$ is a $k$-dimensional subspace. If $T : W \to V$ is simple, then the number of extensions of $T$ to all of $V$ with characteristic polynomial $f$ is given by $\prod_{j = k+1}^{n-1} (q^n -q^j). $
\end{lem}

\begin{lem}
Let $A \in \End(W)$ and $U$ a nonzero proper subspace of $W$ of dimension $r$, so that $1 \leq r < m$. Define
$$
\CalF_{A,U} = \Set{ B \in \CalC_h : (A-B)U \sub U}
$$
Then $\# \CalF_{A,U} \leq q^{m^2-m-r(m-r)}$.
\end{lem}

\begin{proof}

Let $\pi: W \to W/U$ be the quotient map. The condition $(A-B)U \subset U$ is equivalent to $\pi \circ B|_U = \pi \circ A|_U$. Indeed, if for all $u \in U$, $(A-B)u \in U$ it follows that $Au = Bu + U$ in $W/U$ so that $\pi \circ B|_U = \pi \circ A|_U$, and conversely. 

Then $T = B|_U : U \to W$ is a lift of the fixed map $\tau : = \pi \circ A|_U : U \to W/U$. If $T_0$ is one lift of $\tau$, every lift is uniquely given by $T = T_0 + \iota \phi$, where $\phi \in \Hom(U,U)$ and $\iota : U \hookrightarrow W$. Indeed, suppose $T_0 : U \to W$ be one fixed lift, so that $\pi \circ T_0 = \tau$ also. Let $T: U \to W$ be a lift such that $\pi \circ T = \tau$. Then $\pi \circ (T - T_0) = 0$ so that for every $u \in U$, we have $(T - T_0)(u) \in \ker \pi = U$. Thus $T-T_0$ is actually a linear map $U \to U$, so that there is some $\phi \in \Hom(U,U)$ such that $T - T_0 = \iota \circ \phi$ where $\iota : U \hookrightarrow W$ is the natural inclusion. Conversely, if $\phi \in \Hom(U,U)$ and we define $T = T_0 + \iota  \phi$, then $\pi \circ T = \pi \circ T_0 + \pi \circ \iota \circ \phi = \pi \circ T_0$ since $\pi \circ \iota = 0$, so that every such $T$ is indeed a lift. As for the uniqueness, if $T_0 + \iota \phi_1 = T_0 + \iota \phi_2$, then $\iota (\phi_1  - \phi_2) = 0$, and since $\iota$ is injective $\phi_1 = \phi_2$. Since $\dim_{F} \End(U) = r^2$, it follows that the number of lifts is exactly $q^{r^2}.$

Now suppose such a $T$ is the restriction of some $B \in \CalC_h$. Then $T$ must be a \textit{simple} partial linear map. Indeed, suppose that $0 \neq U_0 \subseteq U$ be a nonzero subspace of $U$ fixed by $T$, i.e. $T(U_0) \subseteq U_0$. Since $T = B|_U$, it follows that $U_0$ is $B$-invariant, contrary to the hypothesis that $B$ has irreducible characteristic polynomial $h$. Thus only simple lifts $T$ can occur.  By Lemma \ref{ram}, the number of extensions $B \in \End(W)$ satisfying $B|_U = T$ with $\chi_B(T) = h$ is $\prod_{j =r +1}^{m-1} (q^m -q^j)$. Since there are at most $q^{r^2}$	possible simple lifts, 
$$
\#F_{A, U} \leq q^{r^2} \prod_{j=r+1}^{m-1}(q^m -q^j) \leq q^{r^2} q^{m(m-r-1)} = q^{m^2-m - r(m-r)}.
$$
\end{proof}

\noindent We need a uniform estimate for the Gaussian binomial coefficient: 
\begin{lem}\label{gauss_bc}
For integers $0 \leq b \leq a$, the Gaussian binomial coefficient satisfies
$$
\begin{bmatrix}a \\ b \end{bmatrix}_q \leq 8 q^{b(a-b)}.
$$
\end{lem}
\begin{proof}
Indeed, by definition, 
$$
\begin{bmatrix}a \\ b \end{bmatrix}_q = q^{b(a-b)} \prod_{i = 0}^{b-1} \frac{1-q^{-(a-i)}}{1- q^{-(b-i)}} \leq K q^{b(a-b)},
$$
where $K = \prod_{j =1}^{\infty} \frac{1}{1- 2^{-j}}$. The product $K$ converges, and since $- \log(1-x) \leq 2x$ for $0 \leq x \leq 1/2$, we have $\log K \leq 2 \sum_{j =1}^{\infty} 2^{-j} = 2$ so that $K \leq e^2 < 8$, and the lemma follows.
\end{proof}

\begin{prop}\label{nonzstable}
Fix $A \in \End(W)$ and $0 \neq c \in W$. Let 
$$
Z_{A,c} = \Set{B \in \CalC_h : c \trm{ is not cyclic for } A - B}.
$$
Then 
$$
\#Z_{A,c} \leq 16 q^{m^2 - m -1}.
$$
Thus $\#Z_{A,c} = O(q^{m^2 - m -1})$ with an absolute implied constant.

\end{prop}

\begin{proof}
Let $B \in Z_{A,c}$ and $S := A-B$. Since $c$ is not cyclic for $S$, its cyclic span $U_B = \ip{c, Sc, S^2c, \ldots}$ is a proper subspace of $W$. Moreover $c \in U_B$ and $S(U_B) \sub U_B$. Thus $(A-B) U_B \sub U_B$, so that 
$$	
Z_{A,c} \subseteq \bigcup_{F c \subseteq U  \subsetneq W} 
\CalF_{A, U}.
$$
Group the subspaces $U$ according to $r= \dim U$. The number of $r$-dimensional subspaces of $W$ containing the fixed line $Fc$ is given by the Gaussian binomial coefficient  $\begin{bmatrix}m-1 \\ r-1\end{bmatrix}_q$, which by Lemma \ref{gauss_bc} is at most $8 q^{(r-1)(m-r)}$ (with $a = m-1$ and $b = r-1$). Thus
\begin{align*}
\#Z_{A,c} &\leq \sum_{r = 1}^{m-1} \begin{bmatrix} m-1 \\ r-1 \end{bmatrix}_q	 q^{m^2-m - r(m-r)} \\
& \leq 8 q^{m^2-m} \sum_{r = 1}^{m-1} q^{(r-1)(m-r) - r(m-r)} \\
&= 8 q^{m^2 -m} \sum_{r = 1}^{m-1} q^{-(m-r)} \leq 8 \cdot q^{m^2 -m} \cdot \frac{2}{q} \\
&= 16 q^{m^2 - m -1}.
\end{align*}

\end{proof}

\begin{lem}\label{nonzstable_count}
Let $z$ be as above. Let $S_z$ be the set of $z$-stable lines. Then $\# S_z \leq 4 q^{n-2}$. 
\end{lem}

\begin{proof}
Let the distinct eigenvalues of $z$ lying in $F$ be $\al_1, \ldots, \al_t$. 
If $t = 0$ then $S_z = \emptyset$, and there is nothing to prove. Hence assume that $t \geq 1$. Let $d_i := \dim \ker(z - \al_iI)$. The $z$-stable $F$-lines are exactly the lines contained in these eigenspaces, so their number is 
$$
 \# S_z = \sum_{i = 1}^t \frac{q^{d_i -1}}{q-1}.
$$
Since $z$ is nonscalar, all $d_i \leq n-1$. Furthermore $\sum_i d_i \leq n$. Now $\frac{q^d-1}{q-1} \leq 2 q^{d-1}$. We claim that $\sum_i q^{d_i -1} \leq 2q^{n-2}$. Indeed, if $t = 1$, this follows immediately from $d_1 \leq n-1$. If $t \geq 2$, put $D_2 = d_2 + \ldots +d_t$. For positive $a,b$ one has the inequality 
$$
q^{a-1} + q^{b-1} \leq 2q^{a+b-2} \leq q^{a+b-1}
$$
Therefore $\sum_{i = 2}^t q^{d_i -1} \leq q^{D_2-1}$. Since $d_1 \leq n-1$, and $D_2 \leq n-1$, we have 
$$
\# S_z \leq 2 \sum_{i = 1}^{t} q^{d_i -1} \leq 2(q^{d_1 -1} + q^{D_2 -1}) \leq  4 q^{n-2}
$$
as desired.
\end{proof}

\subsection{Main Theorem}

We now prove the main result Theorem \ref{mainthm1}.

\begin{proof}
We already noted in the introduction, \S \ref{intro}, that $N_{X, Y}(z)= 0$ if $z$ is a scalar. So assume that $z$ is not a scalar. For $y \in Y$, let $L_y = \Ker(y - \lam I)$ be the (unique) $\lam$-eigenline of $y$. Consider the map 
\begin{align*}
\pi_z: \Set{ (x,y) \in X \times Y : x + y = z}  &\to \P(V)\\
(x,y) &\mapsto L_y
\end{align*}
Then
$$
\Set{ (x,y) \in X \times Y : x + y = z} = \bigsqcup_{L \in \P(V)} \pi_z^{-1}(L),
$$
so that  
$$
N_{X, Y}(z) = \sum_{L \in \P(V)} \# \pi_z^{-1}(L).
$$
If $L_y$ is $z$-stable, then $\pi_z^{-1}(L) = \emptyset$.  Indeed, if $L_y$ is $z$-stable then $c = 0$ so that $c$ is not a cyclic vector for $A-B$ hence by Proposition \ref{cyclic_noncyclic}, $z-y \not \in X$ for every $r \in W^*$, whence the fiber $\pi_z^{-1}(L)$ is empty. 

So in the summation for $N_{X, Y}(z)$ we may sum over lines $L$ which are not $z$-stable. For such lines, $c \neq 0$, and by Proposition \ref{cyclic_noncyclic},
\begin{align*}
\#\pi_z^{-1}(L) &= \# \Set{B \in \CalC_h : c \trm{ is cyclic for } A - B}\\
& = \# \CalC_h - \#Z_{A,c}
\end{align*}
where  $Z_{A,c}:= \Set{ B \in \CalC_h : c \trm{ is not cyclic for } A - B}.$
By Proposition \ref{nonzstable} and Lemma \ref{m_conj_count}, we see that 
\begin{align*}
\abs{\# \pi_z^{-1}(L)  - q^{m^2 -m}} &= \abs{ \# \CalC_h - \# Z_{A,c} - q^{m^2 -m}} \\
&\leq \abs{\#\CalC_h - q^{m^2-m}} + \#Z_{A,c} \\
& \leq 2 q^{m^2-m-1} + 16 q^{m^2 - m-1} = 18q^{m^2 - m -1}.
\end{align*}
\noindent The number of $F$-lines in $V$  is $\#\P(V) = \frac{q^n-1}{q-1} = q^{n-1} + q^{n-2} + \ldots +1$. The lower-order terms satisfy $q^{n-2} + \ldots + 1 \leq 2 q^{n-2}$. Therefore, if $G_z$ denotes the number of non-$z$-stable lines, 
$$
G_z = \frac{q^n -1 }{q-1} - \#S_z
$$
where $S_z$ is as in Lemma \ref{nonzstable_count}. Thus $\abs{G_z - q^{n-1}} \leq 6 q^{n-2}$. Furthermore, $G_z \leq \frac{q^n - 1}{q-1} \leq 2 q^{n-1}$.
 
From the above, for every non-$z$-stable line $L$, we have $\abs{\# \pi_z^{-1}(L)  - q^{m^2 -m}} \leq 18q^{m^2 - m -1}$, so that $\#\pi_z^{-1}(L) = q^{m^2 - m} + \eps_L$, where $\abs{\eps_L} \leq 18 q^{m^2 - m -1}$. Since $z$-stable lines contribute zero to the sum, we can rewrite 
$$
N_{X, Y}(z) = \sum_{L \trm{ nonstable }} (q^{m^2 - m} + \eps_L) = G_z q^{m^2 - m} + \sum_{L \trm{ nonstable }} \eps_L.
$$
Since $\abs{G_z - q^{n-1}} \leq 6 q^{n-2}$, it follows that 
$$
\abs{G_z q^{m^2 - m} - q^{n-1 + m^2 - m}} \leq 6 q^{n-2 + m^2 - m}.
$$
But $n-1 + m^2 - m = (n-1) + (n-1) (n-2) = (n-1)^2$, so that 
$$
\abs{G_z q^{m^2 - m} - q^{(n-1)^2}} \leq 6 q^{(n-1)^2 - 1}.
$$
For the second term
\begin{align*}
\big|\sum_{L \trm{ nonstable} } \eps_L \big| &\leq  G_z \cdot 18 q^{m^2 - m -1} \\
&\leq 2 q^{n-1}  \cdot 18 q^{m^2 - m -1} \\
&= 36 q^{(n-1)^2 -1}
\end{align*}
Combining these inequalities, we get the desired estimate.

\end{proof}

\begin{lem}\label{tracerealization}
Let $n \geq 2$, let $q$ be a prime power, and let $t \in F$. Assume $(q,n,t) \neq (2,2,1)$. Then there exist conjugacy classes $X, Y \sub M_n(F)$ such that $\chi_X(T)$ is irreducible of degree $n$, and $\chi_Y(T) = (T - \lambda) h(T)$, where $h(T)$ is irreducible of degree $n-1$ and $h(\lambda) \neq 0$; when $n = 2$, the two linear factors are distinct. Moreover, $\tr(X) + \tr(Y) = t$. In fact, the only genuine exceptional situation is when $q = n = 2$ with $t = 1$.
\end{lem}

\begin{proof}
First suppose that $n \geq 3$. Choose any monic irreducible polynomial $f(T) \in F[T]$ of degree $n$, and any monic irreducible polynomial $h(T) \in F[T]$ of degree $n-1$. Recall that the monic irreducible polynomials exist in every positive degree over $F$ \cite{CM}. 

Let $f(T) = T^n - a T^{n-1} + \ldots$ and $h(T) = T^{n-1} - b T^{n-2} + \ldots$, so that the matrix having characteristic polynomial $f$ has trace $a$, and a matrix having characteristic polynomial $h$ has trace $b$. 

Now define $\lambda := t - a - b$. Since $n \geq 3$, $\deg h = n-1 \geq 2$. Since $h$ is irreducible of degree at least $2$, it has no root in $F$, so that $h(\lambda) \neq 0$. Let $C_f$ be the companion matrix of $f$ and $C_h$ be the companion matrix of $h$. Set $X = \Set{g C_f g^{-1} : g \in \GL_n(F)}	 \sub M_n(F)$ be the conjugacy class of $C_f$ and $Y$ be the conjugacy class of $y_0 = \begin{pmatrix} \lambda & 0  \\ 0 & C_h \end{pmatrix}$. Then $\chi_X(T) = f$ is irreducible of degree $n$, while $\chi_Y(T) = (T- \lambda) h(T)$. Also $\Tr(X) = a$ and $\tr(Y) = \lambda + b$, so that $\tr(X) + \tr(Y) = a + \lam + b = t$.

Now we deal with the case when $n = 2$.  Now both factors of $Y$ are linear and we require them to be distinct. We deal with odd and even $q$ separately. 

\noindent \textbf{Case 1: $q$ odd, $n =2$: } Let $d \in F^{\times}$ be a nonsquare, so that $f(T) = T^2 -d $ is irreducible over $F$ with trace $0$. Thus, we need a split polynomial with distinct roots $\lam, \mu$ such that $\lam + \mu = t$. Choose any nonzero $\delta$, and put $\lam = \frac{t}{2} + \delta$ and $\mu = \frac{t}{2}- \delta$ so that their sum is $t$, and $\lam -\mu = 2 \delta \neq 0$ since $q$ is odd. Take $X$ be the conjugacy class with characteristic polynomial $T^2 - d$ and $Y$ be the conjugacy class of $\begin{pmatrix} \lam & 0 \\ 0 & \mu \end{pmatrix}$. Then $\tr(X) =0, \tr(Y) = \lam + \mu = t$ and their sum $\tr(X) + \tr(Y) = t$ as desired. 

\noindent \textbf{Case $2(i)$: $q > 2$ even, $n = 2$:} First we claim that for every $a \in F^{\times}$ there is an irreducible polynomial over $F$ having trace $a$. Indeed, the $\F_2$-linear map $L : F \to F$ defined by $L(u) = u^2 + u$ has kernel $\Set{0,1}$, because $u^2 +u = 0$ implies $u(u+1) = 0$.  Hence the cardinality of the image $\textrm{Im}(L)$ is $q/2 < q$. Let $c \not \in \textrm{Im}(L)$, and consider the polynomial $U^2 + U + c$, which has no root in $F$, so that it is irreducible.  

Now, for any $a \neq 0$, consider $f_a(T) = T^2 + aT + a^2 c$. If $x \in F$ were a root, dividing by $a^2$ and putting $u = x/a$ would give $u^2 + u + c = 0$ or $c = u^2 + u$ contrary to the choice of $c$. Hence $f_a$ is irreducible, and it has trace $a$; the reader may refer to Artin-Schreier kind of argument in \cite{Gl}.

Because $q > 2$, we can choose $a \in F^{\times}$ such that $a \neq t$ (with $t \in F$, the required trace). Choose an irreducible polynomial $f_a$ having trace $a$ as in the previous paragraph, and put $c_0 = t - a \neq 0$. Let $\mu =0$ and $\lam = c_0$. Then $\lam \neq \mu$ and $\lam + \mu = c_0 = t -a$. Choose $X$ to be conjugacy class with characteristic polynomial $f_a$ and $Y$ be the conjugacy class of $\begin{pmatrix} c_0 & 0 \\ 0 & 0 \end{pmatrix}$, so that $\tr(X) = a$ and $\tr(Y) = c_0 = t- a$ and their sum $\tr(X) + \tr(Y) = t$, proving the result in the case $q$ is even and $q > 2$.

\noindent \textbf{Case $2(ii)$: $q = 2$, $n = 2$:}
The only monic irreducible quadratic polynomial over $\F_2$ is $T^2 + T+ 1$. Thus, necessarily $\tr(X) = 1$. As for $Y$, there are only two distinct linear factors $T, T-1$ so that $\chi_Y(T)$ must necessarily be $T(T-1)$, whence $\tr(Y) = 1$. Consequently, $\tr(X) + \tr(Y) = 0$. Therefore if $t = 0$, the required pair $X, Y$  does  exist, otherwise when $t =1$, the pair does not exist. 
\end{proof}

Now we can prove Corollary \ref{cor2}
\begin{proof}
This follows from Lemma \ref{tracerealization} and Theorem \ref{mainthm1}.
\end{proof}

\bibliographystyle{elsarticle-num} 

\begin{thebibliography}{00}



\bibitem{CM} Sunil K. Chebolu, Jan Minac: Counting irreducible polynomials over finite fields using the inclusion-exclusion principle, \url{https://doi.org/10.48550/arXiv.1001.0409.} 

\bibitem{Gl} S.P. Glasby: Hilbert's Theorem $90$, periodicity, and roots of Artin-Schreier polynomials,  \url{https://doi.org/10.48550/arXiv.2505.00346}



\bibitem{Ki} K. Kishore: Matrix Waring Problem, Linear Algebra and its Applications, Vol. 646, (2022), 84--94. 

\bibitem{KS} K. Kishore, A. Singh: Matrix Waring problem II, Isr. J. Math. 267 (2025), 301–320.  

\bibitem{KVZ} K. Kishore, A. Vasiu, S. Zhan, Waring Problem for Matrices over Finite Fields, J. Pure and Appl. Alg. 228 (2024), 107656.

\bibitem{LST}  M.Larsen, A. Shalev, P.H. Tiep: The Waring problem for finite simple groups, Annals of Mathematics, Vol. 174, Issue 3, (2011), 1885--1950

\bibitem{LW} S. Lang, A. Weil: Number of points of varieties in finite fields: American Journal of Mathematics, Vol. 76, No. 4, (1954), 819--827.


\bibitem{Ra} S. Ram: Simple operators and $q$-Whittaker coefficients of power sum symmetric functions, \url{https://doi.org/10.48550/arXiv.2411.16485}


\bibitem{Sh} Shalev, Aner: Word maps, conjugacy classes, and a noncommutative Waring-type theorem, Ann. of Math. (2) 170 (2009), no. 3, 1383–– 1416.



\end{thebibliography}

\end{document}